\documentclass[12pt]{amsart}
\usepackage{amsthm}
\usepackage{amssymb}
\usepackage{geometry}
\usepackage{enumerate}
\usepackage{setspace}
\usepackage{upgreek}
\usepackage{tikz}
\usepackage{tikz-cd}
\usetikzlibrary{matrix, arrows}
\usepackage[unicode=true]
 {hyperref}
\hypersetup{
colorlinks=true,
urlcolor=black,
citecolor=blue,
linkcolor=blue,
}

\allowdisplaybreaks

\newtheorem{thm}{Theorem}[section]
\newtheorem{prop}[thm]{Proposition}
\newtheorem{lem}[thm]{Lemma}

\theoremstyle{definition}
\newtheorem{definition}[thm]{Definition}
\newtheorem{example}[thm]{Example}
\newtheorem{question}[thm]{Question}
\newtheorem{notation}[thm]{Notation}

\theoremstyle{remark}
\newtheorem{remark}[thm]{Remark}

\numberwithin{equation}{section}

\newcommand{\Frob}{\mathrm{Frob}}
\newcommand{\Hom}{\mathrm{Hom}}
\newcommand{\Gal}{\mathrm{Gal}}
\newcommand{\Cl}{\mathrm{Cl}}
\newcommand{\sd}{\mathrm{\tiny sd}}
\newcommand{\cyc}{\mathrm{\tiny cyc}}

\newcommand{\ep}{\epsilon}
\newcommand{\bZ}{\mathbb{Z}}
\newcommand{\bQ}{\mathbb{Q}}
\newcommand{\A}{\mathcal{A}}
\newcommand{\Image}{\mathrm{Im}}
\newcommand{\Tr}{\mathrm{Tr}}

\begin{document}

\large 

\title[Self-duality of rings of integers in tame extensions]{On the self-duality of rings of integers\\ in tame and abelian extensions}

\author{Cindy (Sin Yi) Tsang}
\address{School of Mathematics, Sun Yat-Sen University, Zhuhai}
\email{zengshy26@mail.sysu.edu.cn}
\urladdr{http://sites.google.com/site/cindysinyitsang/} 

\date{\today}

\begin{abstract}Let $L/K$ be a tame and Galois extension of number fields with group $G$. It is well-known that any ambiguous ideal in $L$ is locally free over $\mathcal{O}_KG$ (of rank one), and so it defines a class in the locally free class group of $\mathcal{O}_KG$, where $\mathcal{O}_K$ denotes the ring of integers of $K$. In this paper, we shall study the relationship among the classes arising from the ring of integers $\mathcal{O}_L$ of $L$, the inverse different $\mathfrak{D}_{L/K}^{-1}$ of $L/K$, and the square root of the inverse different $A_{L/K}$ of $L/K$ (if it exists), in the case that $G$ is abelian. They are naturally related because $A_{L/K}^2 = \mathfrak{D}_{L/K}^{-1} = \mathcal{O}_L^*$, and $A_{L/K}$ is special because $A_{L/K} =  A_{L/K}^*$, where $*$ denotes dual with respect to the trace of $L/K$. 
\end{abstract}

\maketitle

\tableofcontents

\newpage

\section{Introduction}\label{intro}

Let $L/K$ be a Galois extension of number fields with group $G$. There are two ambiguous ideals in $L$, namely ideals in $L$ which are invariant under the action of $G$, whose Galois module structure has been studied extensively in the literature. The first is the ring of integers $\mathcal{O}_L$ of $L$, the study of which is a classical problem; see \cite{Frohlich}. The second is the square root $A_{L/K}$ (if it exists) of the inverse different ideal $\mathfrak{D}_{L/K}^{-1}$ of $L/K$, the study of which was initiated by B. Erez in \cite{Erez}. By Hilbert's formula \cite[Chapter IV, Proposition 4]{Serre}, this ideal $A_{L/K}$ exists when $|G|$ is odd, for example. Also, we note that $A_{L/K}$ is special because it is the unique ideal in $L$ (if it exists) which is self-dual with respect to the trace $\Tr_{L/K}$ of $L/K$.

\vspace{1mm}

It is natural to ask whether the Galois module structures of $\mathcal{O}_L$ and $A_{L/K}$ coincide. More specifically, suppose that $L/K$ is tame. Then, any ambiguous ideal $\mathfrak{A}$ in $L$ is locally free over $\mathcal{O}_KG$ of rank one  by \cite[Theorem 1]{Ullom}. Hence, it determines a class $[\mathfrak{A}]_{\bZ G}$ in $\Cl(\bZ G)$ as well as a class $[\mathfrak{A}]$ in $\Cl(\mathcal{O}_KG)$, where $\Cl(-)$ denotes locally free class group. Provided that $A_{L/K}$ exists, we ask:

\begin{question}\label{Q1} Does $[\mathcal{O}_L]_{\bZ G} = [A_{L/K}]_{\bZ G}$ hold in $\Cl(\bZ G)$?
\end{question}

\begin{question}\label{Q2} Does $[\mathcal{O}_L] = [A_{L/K}]$ hold in $\Cl(\mathcal{O}_KG)$?
\end{question}

Since $A_{L/K}$ is self-dual with respect to $\Tr_{L/K}$ and $\mathfrak{D}_{L/K}^{-1}$ is the dual of $\mathcal{O}_L$ with respect to $\Tr_{L/K}$ by definition, we have that
\begin{align*}[\mathcal{O}_L]_{\bZ G}= [A_{L/K}]_{\bZ G}  &\mbox{ implies }[\mathcal{O}_L]_{\bZ G} = [\mathfrak{D}_{L/K}^{-1}]_{\bZ G},\\
 [\mathcal{O}_L] = [A_{L/K}]& \mbox{ implies }[\mathcal{O}_L] = [\mathfrak{D}_{L/K}^{-1}].\end{align*}
In other words, for Questions~\ref{Q1} and~\ref{Q2} to admit an affirmative answer, the ideal $\mathcal{O}_L$ is necessarily \emph{stably self-dual} as a $\bZ G$-module and an $\mathcal{O}_KG$-module, respectively. It is then natural to also ask:

\begin{question}\label{Q3} Does $[\mathcal{O}_L]_{\bZ G} = [\mathfrak{D}_{L/K}^{-1}]_{\bZ G}$ hold in $\Cl(\bZ G)$?
\end{question}

\begin{question}\label{Q4}Does $[\mathcal{O}_L] = [\mathfrak{D}_{L/K}^{-1}]$ hold in $\Cl(\mathcal{O}_KG)$?
\end{question}

On the one hand, a theorem of M. J. Taylor \cite{Taylor1} implies that Question~\ref{Q3} admits an affirmative answer; this fact was re-established by S. U. Chase \cite{Chase}. Using tools from \cite{Chase}, L. Caputo and S. Vinatier showed in \cite{CV} that Question~\ref{Q1} also admits an affirmative answer as long as $L/K$ is locally abelian.

\vspace{1mm}

On the other hand, both Questions~\ref{Q2} and~\ref{Q4} have never been considered in the literature. The main purpose of this paper is to show that for $K\neq\bQ$, they both admit a negative answer in general; see Theorem~\ref{thm main} below.

\subsection{Basic set-up and notation} Fix a number field $K$ as well as a finite group $G$. Let us define
\begin{align*}
R(\mathcal{O}_KG) & =  \{[\mathcal{O}_L] :\text{tame $L/K$ with }\Gal(L/K)\simeq G\},\\
R_{\sd}(\mathcal{O}_KG) & =  \{[\mathcal{O}_L] :\text{tame $L/K$ with }\Gal(L/K)\simeq G\mbox{ and }[\mathcal{O}_L] = [\mathfrak{D}_{L/K}^{-1}]\},
\end{align*}
where ``sd'' stands for ``self-dual''. For $G$ of odd order, further define
\[ \A^t(\mathcal{O}_KG) = \{[A_{L/K}] :\text{tame $L/K$ with }\Gal(L/K)\simeq G\}.\]
Let us remark that both classes $[\mathcal{O}_L]$ and $[A_{L/K}]$ depend upon the choice of the isomorphism $\Gal(L/K)\simeq G$. For $K\neq\bQ$, we shall prove that even the weakened versions of Questions~\ref{Q2} and~\ref{Q4} below admit a negative answer in general; see Theorem~\ref{thm main} below.

\begin{question}\label{Q5}Does $R_{\sd}(\mathcal{O}_KG) = \A^t(\mathcal{O}_KG)$ hold when $|G|$ is odd?
\end{question}

\begin{question}\label{Q6}Does $R(\mathcal{O}_KG) = R_{\sd}(\mathcal{O}_KG)$ hold?
\end{question}

In what follows, for simplicity, suppose that $G$ is abelian. We shall implicitly suppose also that $G$ has odd order whenever we write $\A^t(\mathcal{O}_KG)$. Then, the three subsets of $\Cl(\mathcal{O}_KG)$ in question are related to the so-called Adams operations on $\Cl(\mathcal{O}_KG)$ as follows; also see \cite{Burns2} and~\cite{BurnsChinburg} for other connections between Adams operations and Galois module structures.

\vspace{1mm}

For each $k\in\bZ$ coprime to $|G|$, the \emph{$k$th Adams operation} is defined by
\[ \Psi_k\in\mbox{Aut}(\Cl(\mathcal{O}_KG));\hspace{1em}\Psi_k([X]) = [X_k],\]
where $X$ denotes an arbitrary locally free $\mathcal{O}_KG$-module of rank one, and $X_k$ denotes the $\mathcal{O}_K$-module $X$ on which $G$ acts via
\[ s* x = \phi_k^{-1}(s)\cdot x \mbox{ for $s\in G$ and $x\in X$},\]
where $\phi_k$ is the automorphism on $G$ given by $\phi_k(s) = s^k$. For example, when $X = \mathcal{O}_L$, where $L/K$ is a tame and Galois extension with $h:\Gal(L/K)\simeq G$, then we have $X_k = \mathcal{O}_{L'}$, where $L'=L$ but with $h':\Gal(L'/K)\simeq G$ defined by $h' = \phi_k\circ h$; similarly when $X = A_{L/K}$. In the case that $k = -1$, we have
\begin{align*} \Psi_{-1}([\mathcal{O}_L]) &= [\Hom_{\mathcal{O}_K}(\mathcal{O}_L,\mathcal{O}_K)]^{-1},\\
\Psi_{-1}([A_{L/K}]) &= [\Hom_{\mathcal{O}_K}(A_{L/K},\mathcal{O}_K)]^{-1},\end{align*}
by \cite[Appendix IX, Proposition 3]{F paper}. Let $*$ denote  dual with respect to $\Tr_{L/K}$. Since $\mathfrak{D}_{L/K}^{-1} = \mathcal{O}_L^*$ and $A_{L/K} = A_{L/K}^*$, we then deduce that
\[ \Psi_{-1}([\mathcal{O}_L]) = [\mathfrak{D}_{L/K}^{-1}]^{-1}\mbox{ and }\Psi_{-1}([A_{L/K}]) = [A_{L/K}]^{-1},\]
where the latter equality was proven in \cite[Theorem 1.2 (a)]{Tsang} as well. In the case that $|G|$ is odd and $k=2$, we further have
\[ [A_{L/K}] = [\mathcal{O}_L]\Psi_2([\mathcal{O}_L]),\]
which was shown in \cite[Theorem 1.2.4]{Tsang thesis} and is also essentially a special case of \cite[Theorem 1.4]{Burns2}.

\vspace{1mm}

Now, it is known by \cite{McCulloh} that $R(\mathcal{O}_KG)$ is a subgroup of $\Cl(\mathcal{O}_KG)$. Writing the operation in $\Cl(\mathcal{O}_KG)$ multiplicatively, we then have well-defined maps
\begin{align*} \Xi_k : R(\mathcal{O}_KG) \longrightarrow R(\mathcal{O}_KG);&\hspace{1em}\Xi_k([X]) = [X]\Psi_k([X]),\\
\Xi_k' : R(\mathcal{O}_KG) \longrightarrow R(\mathcal{O}_KG);&\hspace{1em}\Xi_k'([X]) = [X]^{-1}\Psi_k([X]),
\end{align*}
which are in fact homomorphisms because $\Cl(\mathcal{O}_KG)$ is an abelian group. In addition, the above discussion implies that
\begin{equation}\label{Xi} R_{\sd}(\mathcal{O}_KG) = \ker(\Xi_{-1})\mbox{ and }\A^t(\mathcal{O}_KG) = \Image(\Xi_2'),\end{equation}
which are hence subgroups of $R(\mathcal{O}_KG)$. In particular, we have a chain
\begin{equation}\label{chain} R(\mathcal{O}_KG)\supset R_{\sd}(\mathcal{O}_KG) \supset \A^t(\mathcal{O}_KG) \end{equation}
of subgroups in $\Cl(\mathcal{O}_KG)$. From (\ref{Xi}), we deduce the following criteria which distinguish classes in these three subgroups.

\begin{prop}\label{criteria}Suppose that $G$ is abelian and let $c\in R(\mathcal{O}_KG)$.
\begin{enumerate}[(a)]
\item Assume that $\Psi_{-1}(c) = c$. Then, we have $c\in R_{\sd}(\mathcal{O}_KG)$ if and only if $|c|$ divides two.
\item Assume that $|G|$ is odd and that $\Psi_2(c) = c$. Then, we have $c\in A^t(\mathcal{O}_KG)$ only if $c^{n_G(2)}=1$, where $n_G(2)$ is the multiplicative order of $2$ mod $|G|$.
\end{enumerate}
\end{prop}
\begin{proof}Part (a) follows directly from (\ref{Xi}). As for part (b), suppose that $|G|$ is odd and  that $c\in A^t(\mathcal{O}_KG)$. By (\ref{Xi}), we know that $c = d^{-1}\Psi_2(d)$ for some $d\in R(\mathcal{O}_KG)$. This implies that
\[ \prod_{j=0}^{n_G(2)-1}\Psi_{2^j}(c) = \prod_{j=0}^{n_G(2)-1}\Psi_{2^j}(d)^{-1}\Psi_{2^{j+1}}(d) = \Psi_{2^0}(d)^{-1}\Psi_{2^{n_G(2)}}(d) = 1.\]
It follows that $c^{n_G(2)}=1$ whenever $\Psi_2(c) = c$ holds.
\end{proof}

For notation, let us also define
\[ \Psi_{\bZ} = \{\Psi_k : k\in\bZ\mbox{ coprime to }|G|\},\]
which is plainly a group isomorphic to $(\bZ/|G|\bZ)^\times$, and
\[\Cl^0(\mathcal{O}_KG) = \ker(\Cl(\mathcal{O}_KG)\longrightarrow \Cl(\mathcal{O}_K)),\]
where the map is that induced by augmentation. Our idea is to use Proposition~\ref{criteria} as well as classes in $\Cl^0(\mathcal{O}_KG)^{\Psi_\bZ}$, namely, classes in $\Cl^0(\mathcal{O}_KG)$ which are invariant under $\Psi_\bZ$, to answer Questions~\ref{Q5} and~\ref{Q6}.

\vspace{1mm}

Finally, for each $n\in\mathbb{N}$, let $C_n$ denote a cyclic group of order $n$, and let $\zeta_n$ denote a primitive $n$th root of unity. Given any multiplicative group $\Gamma$, write $\Gamma^n$ for the set of $n$th powers of elements in $\Gamma$.

\subsection{Statements of the main theorems} First, consider $G = C_p$, where $p$ is an odd prime. For $K\neq\bQ$, in order to answer Questions~\ref{Q5} and~\ref{Q6} in the negative, by (\ref{chain}), we must exhibit non-trivial classes in $R(\mathcal{O}_KC_p)$. This was done in \cite{GRRS} and a key ingredient is the inclusion
\begin{equation}\label{GRRS}
(\Cl^0(\mathcal{O}_KC_p)^{\Psi_\bZ})^{(p-1)/2}\subset R(\mathcal{O}_KC_p)^{\Psi_{\bZ}}.\end{equation}
This was shown in the proof \cite[Proposition 4]{GRRS} using the characterization of $R(\mathcal{O}_KC_p)$ due to L. R. McCulloh in \cite{McCulloh0}. Using (\ref{GRRS}), we deduce that:

\begin{prop}\label{criteria'}Let $p$ be an odd prime and let $c\in \Cl^0(\mathcal{O}_KC_p)^{\Psi_\bZ}$.
\begin{enumerate}[(a)]
\item If $|c|$ does not divide $p-1$, then $c^{(p-1)/2}\in R(\mathcal{O}_KC_p)\setminus R_{\sd}(\mathcal{O}_KC_p)$.
\item If $|c|=2$ and $p\equiv-1\pmod{8}$, then $c\in R_{\sd}(\mathcal{O}_KC_p)\setminus \A^t(\mathcal{O}_KC_p)$.
\end{enumerate}
\end{prop}
\begin{proof}Observe that $c^{(p-1)/2} \in R(\mathcal{O}_KC_p)$ by (\ref{GRRS}). Part (a) is then clear from Proposition~\ref{criteria} (a). As for part (b), suppose that $|c|=2$. If $p\equiv-1\pmod{4}$, then $c = c^{(p-1)/2}\in R_{\sd}(\mathcal{O}_KC_p)$ by Proposition~\ref{criteria} (a). If $p\equiv-1\pmod{8}$ in addition, then $c \notin \A^t(\mathcal{O}_KC_p)$ by Proposition~\ref{criteria} (b), because in this case $2$ is a square mod $p$ but $-1$ is not, whence $n_{C_p}(2)$ is necessarily odd.
\end{proof}


Using Proposition~\ref{criteria'} and some further ideas from (\ref{GRRS}), we shall prove:

\begin{thm}\label{thm main}Suppose that $K\neq\bQ$. Then we have:
\begin{enumerate}[(a)]
\item $R(\mathcal{O}_KC_p)\supsetneq R_{\sd}(\mathcal{O}_KC_p)$ for infinitely many odd primes $p$.
\item $R_{\sd}(\mathcal{O}_KC_p)\supsetneq \A^t(\mathcal{O}_KC_p)$ for infinitely many odd primes $p$.
\end{enumerate}
\end{thm}
\begin{proof}We shall prove part (b) in Subsection~\ref{proof sec1}. For part (a), we may deduce it using results in \cite{GRRS} as follows. Let $p$ be an odd prime. Let $T(\mathcal{O}_KC_p)$ denote the Swan subgroup of $\Cl(\mathcal{O}_KC_p)$; see \cite{Ullom Swan} or \cite[Section 53]{CR} for the definition. Then, as shown in the proof of \cite[Proposition 4]{GRRS}, we have
\begin{equation}\label{T in R}
T(\mathcal{O}_KC_p)\subset\Cl^0(\mathcal{O}_KC_p)^{\Psi_\bZ}.
\end{equation}
Using Chebotarev's density theorem, it was further shown in \cite[Theorem 5 and Proposition 9]{GRRS} that $T(\mathcal{O}_KC_p)$ contains a class of order coprime to $p-1$ for infinitely many $p$. The claim now follows from Proposition~\ref{criteria'} (a).
\end{proof}

Since the proof of Theorem~\ref{thm main} uses Chebotarev's density theorem, it does not give explicit primes $p$ satisfying the conclusion. In the special case that $K/\bQ$ is abelian with $K$ imaginary, by slightly modifying the proof, we shall give explicit primes $p$ such that $R_{\sd}(\mathcal{O}_KC_p)\supsetneq\A^t(\mathcal{O}_KC_p)$. See \cite{Herreng} for explicit conditions on $K$, in which $p$ is ramified, such that the $p$-rank of $T(\mathcal{O}_KC_p)$ is at least one, so $R(\mathcal{O}_KC_p)\supsetneq R_{\sd}(\mathcal{O}_KC_p)$ by (\ref{T in R}) and Proposition~\ref{criteria'} (a).

\begin{thm}\label{thm explicit}Suppose that $K/\bQ$ is abelian with $K$ imaginary, and let $m$ be the conductor of $K$. \hspace{-1mm}Then, we have $R_{\sd}(\mathcal{O}_KC_p)\supsetneq\A^t(\mathcal{O}_KC_p)$ for all primes $p$ satisfying $p\equiv-1\pmod{8}$ and $p\equiv-1\pmod{2m}$.
\end{thm}

\begin{example}Consider the special case when $K = \bQ(\sqrt{D})$, where $D$ is a negative square-free integer not divisible by $p$. For simplicity, let us assume that $D\not\in\{-1,-3\}$. Then, by \cite[Lemma 3.2 and Theorem 3.4]{Kobayashi}, we have
\[ T(\mathcal{O}_KC_p)\simeq \begin{cases}
C_{(p+1)/2}\mbox{ or }C_{p+1}&\mbox{if $\left(\frac{D}{p}\right)=-1$},\\
C_{(p-1)/2}\mbox{ or }C_{p-1}&\mbox{if $\left(\frac{D}{p}\right)=1$},\end{cases}\]
where $\left(\frac{\cdot}{\cdot}\right)$ denotes the Legendre symbol. From Proposition~\ref{criteria'} and (\ref{T in R}), we then deduce that
\[\label{quad eg}\begin{cases}
R(\mathcal{O}_KC_p)\supsetneq R_{\sd}(\mathcal{O}_KC_p) &\mbox{if $\left(\frac{D}{p}\right)=-1$ and $p\neq3$},\\
R_{\sd}(\mathcal{O}_KC_p)\supsetneq \A^t(\mathcal{O}_KC_p)&\mbox{if $\left(\frac{D}{p}\right)=-1$ and $p\equiv-1\hspace{-3mm}\pmod{8}$},
\end{cases}\]
where the second statement may be viewed as a refinement of Theorem~\ref{thm explicit}. To see why, note that by quadratic reciprocity, we have
\begin{equation}\label{QR}\left(\frac{2}{p}\right) = (-1)^{\frac{p^2-1}{8}},\,\ \left(\frac{-1}{p}\right) = (-1)^{\frac{p-1}{2}},\,\ \left(\frac{q}{p}\right)\left(\frac{p}{q}\right)=(-1)^{\frac{p-1}{2}\frac{q-1}{2}},\end{equation}
for any odd prime $q$. Suppose that $p\equiv-1\pmod{8}$ and $p\equiv-1\pmod{2m}$, where $m$ is the conductor of $K$. Since $|D|$ divides $m$, we see that any of its prime divisor is a square mod $p$. It follows that
\[ \left(\frac{D}{p}\right) = \left(\frac{-1}{p}\right)\left(\frac{|D|}{p}\right) = -1,\mbox{ whence }R_{\sd}(\mathcal{O}_KC_p)\supsetneq \A^t(\mathcal{O}_KC_p)\]
by the above, as predicted by Theorem~\ref{thm explicit}. Let us note that not much may be deduced from Proposition~\ref{criteria'} if $\left(\frac{D}{p}\right)=1$, and that the case $D\in\{-1,-3\}$ may be dealt with analogously. 
\end{example}

\vspace{-1.75mm}

Next, we return to an arbitrary abelian group $G$. Recall that the proof of Theorem~\ref{thm main} (a) uses the Swan subgroup $T(\mathcal{O}_KG)$ of $\Cl(\mathcal{O}_KG)$. The connect\-ion between Question~\ref{Q6} and the Swan subgroup was already observed in \cite{Chase} and \cite{Taylor2}; they both used the fact that $T(\bZ C)=1$ for all finite cyclic groups $C$ to answer Question~\ref{Q3} in the positive. We shall investigate this connection further as follows. 

\vspace{1mm}

Observe that the first equality in (\ref{Xi}) implies that 
\begin{equation}\label{Rsd}R(\mathcal{O}_KG)/R_{\sd}(\mathcal{O}_KG)\simeq \Image(\Xi_{-1}).\end{equation}
Thus, it suffices to understand $\Image(\Xi_{-1})$. In Subsection~\ref{Swan section}, for each subgroup $H$ of $G$, we shall define a \emph{generalized Swan subgroup} $T_H^*(\mathcal{O}_KG)$ of $\Cl(\mathcal{O}_KG)^{\Psi_\bZ}$, such that $T_G^*(\mathcal{O}_KG)$ is the usual Swan subgroup $T(\mathcal{O}_KG)$. We shall give lower and upper bounds for $\Image(\Xi_{-1})$ in terms of these $T_H^*(\mathcal{O}_KG)$.  

\begin{thm}\label{thm Swan1}Suppose that $G$ is abelian. Let $H$ be a cyclic subgroup of $G$ and let $n$ denote its order.
\begin{enumerate}[(a)]
\item We have $T_H^*(\mathcal{O}_KG)^{d_n(K)} \subset R(\mathcal{O}_KG)^{\Psi_{\bZ}}$, where
\[ d_n(K) = \begin{cases} [K(\zeta_n):K]/2 & \mbox{when $(\zeta_n\mapsto\zeta_n^{-1})\in\Gal(K(\zeta_n)/K)$},\\ [K(\zeta_n):K] & \mbox{when $(\zeta_n\mapsto\zeta_n^{-1})\notin\Gal(K(\zeta_n)/K)$}.\end{cases}\]
In particular, we have  $T_H^*(\mathcal{O}_KG)^{2d_n(K)} \subset \Image(\Xi_{-1})$.
\item We have $T_H^*(\mathcal{O}_KG)\subset \Image(\Xi_{-1})$ if $n$ is odd and $\zeta_n\in K^\times$.
\end{enumerate}
\end{thm}

\begin{thm}\label{thm Swan2}Suppose that $G$ is abelian.
\begin{enumerate}[(a)]
\item We have $\Image(\Xi_{-1}) \subset T_{\cyc}^*(\mathcal{O}_KG)$ if $\Cl(\mathcal{O}_K) = 1$, where
\[T_{\cyc}^*(\mathcal{O}_KG) = \prod_{\substack{H\leq G\\H\mbox{\tiny cyclic}}}T_H^*(\mathcal{O}_KG).\]
\item We have $\Image(\Xi_{-1})\neq1$ if $\Cl(\mathcal{O}_K)^{\delta(G)}\neq1$ and $\zeta_{\exp(G)}\in K^\times$, where
\[ \delta(G) = \begin{cases} 2 & \mbox{when $|G|$ is a power of two},\\ 1 & \mbox{otherwise},\end{cases}\]
and $\exp(G)$ denotes the exponent of $G$, provided that $G\neq1$.
\end{enumerate}
\end{thm}

From Theorems~\ref{thm Swan1} and~\ref{thm Swan2}, as well as (\ref{Rsd}), we deduce that
\[ R(\mathcal{O}_KG) = R_{\sd}(\mathcal{O}_KG)\mbox{ if and only if }\Cl(\mathcal{O}_K) = 1 \mbox{ and }T_{\cyc}^*(\mathcal{O}_KG) =1,\]
under the assumption that $G$ is an abelian group of odd order such that all $|G|$th roots of unity are contained in $K$.

\begin{example}Suppose that $G=C_p$, where $p$ is an odd prime. Applying Theorem~\ref{thm Swan1} (a) to the full group $G$, we obtain
\[ T(\mathcal{O}_KC_p)^{d_p(K)} \subset R(\mathcal{O}_KC_p)^{\Psi_\bZ},\mbox{ where }d_p(K)\mbox{ divides }(p-1)/2,\]
and so we may regard Theorem~\ref{thm Swan1} (a) as a refinement of (\ref{GRRS}) and (\ref{T in R}). By Theorem~\ref{thm Swan2} (a), when $\Cl(\mathcal{O}_K)=1$, we then have a chain 
\[ T(\mathcal{O}_KC_p)^{p-1}\subset T(\mathcal{O}_KC_p)^{2d_p(K)} \subset \Image(\Xi_{-1}) \subset T(\mathcal{O}_KC_p)\]
of inclusions. Let us consider a few special examples of $K$ with $\Cl(\mathcal{O}_K)=1$.

\vspace{1mm}

By \cite[Lemma 3.2 and Theorem 3.4]{Kobayashi}, we know that
\[
T(\mathcal{O}_KC_p) \simeq \begin{cases} C_{(p+1)/4}
&\mbox{if $K=\bQ(\sqrt{-1})$ and $p\equiv3\hspace{-3mm}\pmod{8}$},\\
C_{(p+1)/6}&\mbox{if $K=\bQ(\sqrt{-3})$ and $p\equiv5\hspace{-3mm}\pmod{12}$}.
\end{cases}\]
By \cite{cyclotomic}, we also know that
\[ T(\mathcal{O}_KC_p)\simeq C_p^{\oplus(p-3)/2}\mbox{ if $K=\bQ(\zeta_p)$ and $p\in\{3,5,7,11,13,17,19\}$}.\]
In all of the above cases, we deduce that $\Image(\Xi_{-1}) = T(\mathcal{O}_KC_p)$, and in particular, from (\ref{Rsd}) we see that the difference between $R(\mathcal{O}_KC_p)$ and $R_{\sd}(\mathcal{O}_KC_p)$ becomes bigger as $p$ increases.
\end{example}

\section{Comparison between $R_{\sd}(\mathcal{O}_KG)$ and $\A^t(\mathcal{O}_KG)$}

In this section, we shall prove Theorems~\ref{thm main} (b) and~\ref{thm explicit}, by using Proposition~\ref{criteria'} (b) to exhibit the existence of a class in $R_{\sd}(\mathcal{O}_KC_p)\setminus\A^t(\mathcal{O}_KC_p)$ for infinitely many odd primes $p$.

\vspace{1mm}

In what follows, let $p$ be any odd prime. Define
\[V_p(\mathcal{O}_K) = \frac{(\mathcal{O}_K/p\mathcal{O}_K)^\times}{\uppi_p(\mathcal{O}_K^\times)},\mbox{ where }\uppi_p:\mathcal{O}_K\longrightarrow\mathcal{O}_K/p\mathcal{O}_K\]
is the natural quotient map. Then, we have a surjective homomorphism
\[T(\mathcal{O}_KC_p)\longrightarrow V_p(\mathcal{O}_K)^{p-1},\]
as shown in \cite[Theorem 5]{GRRS}. This, together with (\ref{T in R}), implies that:

\begin{lem}\label{lem1} If $p\equiv-1\pmod{4}$ and $V_p(\mathcal{O}_K)$ has an element of order four, then $\Cl^0(\mathcal{O}_KC_p)^{\Psi_\bZ}$ has an element of order two.
\end{lem}

In the case that $K$ is not totally real, we shall prove Theorem~\ref{thm main} (b) using Lemma~\ref{lem1}. In the case that $K$ is totally real, however, our method fails in general; see Remark~\ref{remark}. Hence, we must look for elements in $\Cl^0(\mathcal{O}_KC_p)^{\Psi_\bZ}$ of order two lying outside of $T(\mathcal{O}_KC_p)$. 

\vspace{1mm}

To that end, let $\mathcal{M}(KC_p)$ denote the maximal order in $KC_p$, and for convenience, assume that $p$ is large enough so that $[K(\zeta_p):K]=p-1$. Then, we have a natural isomorphism
\[ \mathcal{M}(KC_p) \longrightarrow\mathcal{O}_K\times\mathcal{O}_{K(\zeta_p)}; \hspace{1em}\sum_{s\in C_p}\alpha_ss\mapsto \left(\sum_{s\in C_p}\alpha_s,\sum_{s\in C_p}\alpha_s\chi(s) \right),\]
where $\chi$ is a fixed non-trivial character on $C_p$. This induces an isomorphism
\[ \Cl(\mathcal{M}(KC_p)) \simeq \Cl(\mathcal{O}_K)\times\Cl(\mathcal{O}_{K(\zeta_p)}).\]
In particular, we have a surjective homormorphism
\[ \Cl^0(\mathcal{O}_KC_p) \longrightarrow \Cl(\mathcal{O}_{K(\zeta_p)}),\]
such that the $\Psi_\bZ$-action on $\Cl^0(\mathcal{O}_KC_p)$ corresponds precisely to the $\Gamma_p$-action on $\Cl(\mathcal{O}_{K(\zeta_p)})$, where $\Gamma_p= \Gal(K(\zeta_p)/K)$. This implies that:

\begin{lem}\label{lem2}If $\Cl(\mathcal{O}_{K(\zeta_p)})^{\Gamma_p}$ has an element of order two, then $\Cl^0(\mathcal{O}_KC_p)^{\Psi_\bZ}$ also has an element of order two.
\end{lem}

To show that $\Cl(\mathcal{O}_{K(\zeta_p)})^{\Gamma_p}$ contains an element of order two, we shall need the following so-called Chevalley's ambiguous class formula.

\begin{prop}\label{Chevalley}Let $F/K$ be a cyclic extension. Let $\Gamma = \Gal(F/K)$ denote its Galois group and let $N_{F/K}:F\longrightarrow K$ denote its norm. Then, we have
\[ |\Cl(\mathcal{O}_F)^{\Gamma}| = |\Cl(\mathcal{O}_K)|\cdot \frac{2^r\prod\limits_{\mathfrak{p}}e_\mathfrak{p}}{[\mathcal{O}_K^\times:\mathcal{O}_K^\times\cap N_{F/K}(F^\times)][F:K]},\]
where $r$ is the number of real places in $K$ which complexify in $F/K$. Here $\mathfrak{p}$ ranges over the prime ideals in $K$ and $e_\mathfrak{p}$ is its ramification index in $F/K$.
\end{prop}
\begin{proof}See \cite[Chapter II, Remark 6.2.3]{Gras}.
\end{proof}

\begin{lem}\label{lem3} If $K\neq\bQ$ is not totally imaginary, with $[K(\zeta_p):K] = p-1$, and $p$ is totally split in $K/\bQ$, then $\Cl(\mathcal{O}_{K(\zeta_p)})^{\Gamma_p}$ has an element of order two.
\end{lem}
\begin{proof}Assume the hypothesis. Let us write $[K:\bQ] = r_1 + 2r_2$, where $r_1$ and $2r_2$, respectively, denote the number of real and complex embeddings of $K$. Applying Proposition~\ref{Chevalley} to the field $F = K(\zeta_p)$, we then obtain
\[ |\Cl(\mathcal{O}_{K(\zeta_p)})^{\Gamma_p}| = |\Cl(\mathcal{O}_K)|\cdot \frac{2^{r_1}(p-1)^{r_1+2r_2-1}}{[\mathcal{O}_K^\times:\mathcal{O}_K^\times\cap N_{K(\zeta_p)/K}(K(\zeta_p)^\times)]}.\]
Indeed, we have $r = r_1$ since $K(\zeta_p)$ is totally imaginary. Further, the prime ideals $\mathfrak{p}$ in $K$ which ramify in $K(\zeta_p)/K$ are precisely those above $p$. Since $p$ is totally split in $K/\bQ$, there are $[K:\bQ]$ such $\mathfrak{p}$, and $e_\mathfrak{p} = [K(\zeta_p):K] = p-1$.

\vspace{1mm}

Now, by the Dirichlet's unit theorem, we know that
\[ \mathcal{O}_K^\times = \langle\ep_0\rangle\times\langle\ep_1\rangle\times\cdots\times\langle\ep_{r_1+r_2-1}\rangle,\]
where $\ep_0$ is a root of unity and $\ep_1,\dots,\ep_{r_1+r_2-1}$ are fundamental units. Hence, we have a natural surjective homomorphism
\[ \prod_{j=0}^{r_1+r_2-1} \frac{\langle\ep_j\rangle}{\langle \ep_j^{p-1}\rangle}\longrightarrow \frac{\mathcal{O}_K^\times}{\mathcal{O}_K^\times\cap N_{K(\zeta_p)/K}(K(\zeta_p)^\times)},\]
and so the order of the quotient group on the right divides
\[ n_0\cdot (p-1)^{r_1+r_2-1},\mbox{ where }n_0 = [\langle\ep_0\rangle : \langle\ep_0^{p-1}\rangle].\]
Notice that $n_0$ divides $p-1$ and that $n_0 = 2$ when $K$ is totally real. We then deduce that $|\Cl(\mathcal{O}_{K(\zeta_p)})^{\Gamma_p}|$ is divisible by
\[ \frac{2^{r_1}(p-1)^{r_1+2r_2-1}}{n_0(p-1)^{r_1+r_2-1}} = \frac{2^{r_1}(p-1)^{r_2}}{n_0} =
2^{r_1}(p-1)^{r_2-1}\left(\frac{p-1}{n_0}\right).\]
By hypothesis, we have $r_1\geq 1$, and $r_1\geq2$ when $r_2=0$. Hence, the number above is always even, and so $\Cl(\mathcal{O}_{K(\zeta_p)})^{\Gamma_p}$ has an element of order two.
\end{proof}

\subsection{Proof of Theorem~\ref{thm main} (b)}\label{proof sec1}
 
Fix an algebraic closure $K^c$ of $K$. Let $\widetilde{K}$ denote the Galois closure of $K$ over $\bQ$ lying in $K^c$ and let $K_4$ denote the field obtained by adjoining to $\widetilde{K}$ all fourth roots of elements in $\mathcal{O}_K^\times$. Notice that $K_4/\bQ$ is a Galois extension. 

\vspace{1mm}

The next lemma is motivated by \cite[Proposition 9]{GRRS} and it allows us to use Chebotarev's density theorem to prove Theorem~\ref{thm main} (b).

\begin{lem}\label{lem frob}Let $\tau\in\emph{Gal}(K^c/\bQ)$ and let $f\in\mathbb{N}$ denote the smallest natural number such that $\tau^f|_K = \emph{Id}_K$.
\begin{enumerate}[(a)]
\item Suppose that
\begin{equation}\label{tau1}f\mbox{ is even},\,\ \tau^f|_{K_4} = \emph{Id}_{K_4},\,\ \tau|_{\bQ(\sqrt{-1})} \neq\emph{Id}_{\bQ(\sqrt{-1})},\,\ \tau|_{\bQ(\sqrt{2})} = \emph{Id}_{\bQ(\sqrt{2})}.\end{equation}
Let $\mathfrak{P}$ be any prime ideal in $K_4(\sqrt{2})$, unramified over $\bQ$, such that
\[\emph{Frob}_{K_4(\sqrt{2})/\bQ}(\mathfrak{P}) = \tau|_{K_4(\sqrt{2})},\]
and let $p\bZ$ be the prime lying below $\mathfrak{P}$. Then, we have $p\equiv-1\pmod{8}$, and the group $V_p(\mathcal{O}_K)$ has an element of order four.
\item Suppose that
\begin{equation}\label{tau2}f =1,\,\ \tau|_{\widetilde{K}} = \emph{Id}_{\widetilde{K}},\,\ \tau|_{\bQ(\sqrt{-1})} \neq\emph{Id}_{\bQ(\sqrt{-1})},\,\ \tau|_{\bQ(\sqrt{2})} = \emph{Id}_{\bQ(\sqrt{2})}.\end{equation}
Let $\mathfrak{P}$ be any prime ideal in $\widetilde{K}(\sqrt{-1},\sqrt{2})$, unramified over $\bQ$, such that
\[\emph{Frob}_{\widetilde{K}(\sqrt{-1},\sqrt{2})/\bQ}(\mathfrak{P}) = \tau|_{\widetilde{K}(\sqrt{-1},\sqrt{2})},\]
and let $p\bZ$ be the prime lying below $\mathfrak{P}$. Then, we have $p\equiv-1\pmod{8}$, and the prime $p$ is totally split in $K/\bQ$.
\end{enumerate}
\end{lem}
\begin{proof}In both parts (a) and (b), we clearly have $p\equiv -1$ (mod $8)$ because
\[p\equiv-1\hspace{-3mm}\pmod{8}\mbox{ if and only if }\begin{cases}
p\mbox{ is inert in }\bQ(\sqrt{-1})\\ p\mbox{ is split in }\bQ(\sqrt{2})
\end{cases} \]
by (\ref{QR}). In part (b), the prime $p$ is totally split in $\widetilde{K}/\bQ$ and hence in $K/\bQ$.

\vspace{1mm}

In part (a), let $\mathfrak{p}_4$ and $\mathfrak{p}$ denote the prime ideals in $K_4$ and $K$, respectively, lying below $\mathfrak{P}$. Note that $f$ is the inertia degree of $\mathfrak{p}$ over $\bQ$, and we have
\[\Frob_{K_4/K}(\mathfrak{p}_4) = \tau^f|_{K_4} = \mbox{Id}_{K_4}.\]
This means that $\mathfrak{p}$ is totally split in $K_4/K$, and so elements in $\mathcal{O}_K^\times$ reduce to fourth powers in $\mathcal{O}_K/\mathfrak{p}$. Hence, we have surjective homomorphisms
\[ V_p(\mathcal{O}_K) \longrightarrow (\mathcal{O}_K/\mathfrak{p})^\times/\uppi_\mathfrak{p}(\mathcal{O}_K^\times)\longrightarrow (\mathcal{O}_K/\mathfrak{p})^\times/((\mathcal{O}_K/\mathfrak{p})^\times)^4,\]
where $\uppi_\mathfrak{p}:\mathcal{O}_K\longrightarrow\mathcal{O}_K/\mathfrak{p}$ is the natural quotient map. But $(\mathcal{O}_K/\mathfrak{p})^\times\simeq C_{p^f-1}$, \par\noindent and $4$ divides $p^f-1$ because $f\geq 2$ is even. It follows that the last quotient group above and in particular $V_p(\mathcal{O}_K)$ has an element of order four.
\end{proof}

\begin{proof}[Proof of Theorem~\ref{thm main} (b)] Let $\sigma_c,\sigma_r:K^c\longrightarrow \mathbb{C}$ be embeddings such that 
\[\sigma_c(K)\not\subset\mathbb{R}\mbox{ and }\sigma_r(K) \subset\mathbb{R},\]
if they exist. Further, define
\[\tau_c = \sigma_c^{-1}\circ\rho \circ \sigma_c\mbox{ and }\tau_r = \sigma_r^{-1}\circ\rho \circ \sigma_r,\]
where $\rho:\mathbb{C}\longrightarrow\mathbb{C}$ denotes complex conjugation. Observe that:
\begin{enumerate}[(i)]
\item If $K$ is not totally real, then $\sigma_c$ exists, and $\tau_c$ satisfies (\ref{tau1}).
\item  If $K$ is totally real, then $\sigma_r$ exists, and $\tau_r$ satisfies (\ref{tau2}).
\end{enumerate}
In both cases, let $p\equiv-1$ (mod $8$) be a prime given as in Lemma~\ref{lem frob}. Then, we deduce from Lemmas~\ref{lem1},~\ref{lem2}, and~\ref{lem3} that $\Cl^0(\mathcal{O}_KC_p)^{\Psi_\bZ}$ has an element of order two. The claim now follows from Proposition~\ref{criteria'} (b) and Chebotarev's density theorem.
\end{proof}

\begin{remark}\label{remark}Suppose that $K$ is a real quadratic field such that its fundamental unit $\ep$ has norm $-1$ over $\bQ$. For any odd prime $p$ which is inert in $K/\bQ$, we then have $\ep^{p+1}\equiv -1$ (mod $p\mathcal{O}_K$), as shown in \cite[(1.0.1)]{IK}. Letting $n_p(\ep)$ denote the multiplicative order of $\ep$ mod $p\mathcal{O}_K$, this implies that
\[ |V_p(\mathcal{O}_K)| = \frac{|(\mathcal{O}_K/p\mathcal{O}_K)^\times|}{|\uppi_p(\mathcal{O}_K^\times)|} = \frac{p^2-1}{n_p(\ep)} = \frac{2(p+1)}{n_p(\ep)}\cdot\frac{p-1}{2}.\]
The first quotient is an odd integer by \cite[Theorem 1.3]{IK}, so then $V_p(\mathcal{O}_K)$ has odd order when $p\equiv-1$ (mod $4$). This means that we cannot use Lemma~\ref{lem frob} (a) to find primes $p\equiv-1$ (mod $8$) such that $V_p(\mathcal{O}_K)$ has an element of order four.
\end{remark}

\subsection{Proof of Theorem~\ref{thm explicit}} First, we need the following group-theoretic lemmas. 

\begin{lem}\label{group1}Let $\Gamma$ be a finite abelian $p$-group, where $p$ is a prime. Let $\Delta$ be any cyclic subgroup of $\Gamma$ whose order is maximal among all cyclic subgroups of $\Gamma$. Then, there exists a subgroup $\Delta'$ of $\Gamma$ such that $\Gamma = \Delta\times\Delta'$.
\end{lem}
\begin{proof}See the proof of \cite[Chapter I, Theorem 8.2]{Lang}, for example.
\end{proof}
%
%

\begin{lem}\label{group2}Let $\Gamma$ be a group isomorphic to $k$ copies of $C_n$, where $k,n\in\mathbb{N}$, and let $\Delta$ be any cyclic subgroup of order $n$. Then, there exists a subgroup $\Delta'$ of $\Gamma$ such that $\Gamma = \Delta\times\Delta'$. Moreover, for any $x\in\Gamma$, there exists a surjective homomorphism from $\Gamma/\langle x\rangle$ to $k-1$ copies of $C_n$.
\end{lem}
\begin{proof}The first claim is a direct consequence of Lemma~\ref{group1} and plainly $\Delta'$ is necessarily isomorphic to $k-1$ copies of $C_n$. The second claim follows as well  because any $x\in \Gamma$ is contained in some cyclic subgroup $\Delta$ of order $n$. 
\end{proof}

\begin{proof}[Proof of Theorem~\ref{thm explicit}] By Proposition~\ref{criteria'} (b) and Lemma~\ref{lem1}, it is enough to show that $V_p(\mathcal{O}_K)$ has an element of order four. 

\vspace{1mm}

Set $d = [K:\bQ]$ and note that $K\subset\bQ(\zeta_m)$ by hypothesis. First, since $K$ is imaginary, by the Dirichlet's unit theorem, we know that
\[\mathcal{O}_K^\times = \langle\ep_0\rangle\times\langle\ep_1\rangle\times\cdots\times\langle\ep_{d/2-1}\rangle,\]
where $\ep_0$ is a root of unity and $\ep_1,\dots,\ep_{d/2-1}$ are fundamental units. Now, the hypothesis $p\equiv-1$ (mod $m$) implies that $p$ is unramified in $\bQ(\zeta_m)/\bQ$ and
\[\Frob_{\bQ(\zeta_m)/\bQ}(p) = \mbox{complex conjugation}.\]
Since $K$ is imaginary, the inertia degree of $p$ in $K/\bQ$ is equal to two, and so
\[(\mathcal{O}_K/p\mathcal{O}_K)^\times\simeq\prod_{\mathfrak{p}\mid p}(\mathcal{O}_K/\mathfrak{p})^\times\simeq C_{p^2-1}\times\cdots\times C_{p^2-1}\hspace{1em}\mbox{($d/2$ copies}).\]
From Lemma~\ref{group2}, we then deduce that there is a surjective homomorphism
\[  \frac{(\mathcal{O}_K/p\mathcal{O}_K)^\times}{\uppi_p(\langle\ep_1,\dots,\ep_{d/2-1}\rangle)}\longrightarrow C_{p^2-1}.\]
Let $\delta = 2$ if $m$ is odd, and $\delta =1$ if $m$ is even. Then, the order of $\langle\ep_0\rangle$ divides $\delta m$, and we see that there are surjective homomorphisms
\[V_p(\mathcal{O}_K)\longrightarrow C_{\frac{p^2-1}{|\langle\ep_0\rangle|}} \longrightarrow C_{\frac{p^2 - 1}{\delta m}}.\]
The last cyclic group has order dividing four, because 
\[\frac{p^2-1}{\delta m} = \left(\frac{p+1}{2\delta m}\right)\cdot 2(p-1),\mbox{ and }p\equiv-1\hspace{-3mm}\pmod{2\delta m}\]
by hypothesis. Thus, indeed  $V_p(\mathcal{O}_K)$ has an element of order four.
\end{proof}

\section{Comparison between $R(\mathcal{O}_KG)$ and $R_{\sd}(\mathcal{O}_KG)$}

In this section, we shall prove Theorems~\ref{thm Swan1} and~\ref{thm Swan2}. A key ingredient is \par\noindent the characterization of $R(\mathcal{O}_KG)$ due to L. R. McCulloh \cite{McCulloh}, which works for all abelian groups $G$; see Subsection~\ref{char section} below. We remark that the proof of (\ref{GRRS}) given in \cite{GRRS} uses his older characterization of $R(\mathcal{O}_KG)$ from \cite{McCulloh0}, which works only for elementary abelian groups $G$.

\vspace{1mm}

In the subsequent subsections, except in Subsection~\ref{Swan section}, we shall assume that $G$ is abelian. We shall further use the following notation.

\begin{notation}Let $M_K$ denote the set of finite primes in $K$. The symbol $F$ shall denote either $K$ or the completion $K_v$ of $K$ at some $v\in M_K$, and
\begin{align*}
\mathcal{O}_F & = \mbox{the ring of integers in $F$},\\
F^c & =\mbox{a fixed algebraic closure of $F$},\\
\mathcal{O}_{F^c} & = \mbox{the integral closure of $\mathcal{O}_F$ in $F^c$},\\
\Omega_F & = \mbox{the Galois group of $F^c/F$}.
\end{align*}
For each $v\in M_K$, we shall regard $K^c$ as lying in $K_v^c$ via a fixed embedding $K^c\longrightarrow K_v^c$ extending the natural embedding $K\longrightarrow K_v$.
\end{notation}

\subsection{Locally free class group} The class group $\Cl(\mathcal{O}_KG)$ admits an idelic description as follows; see \cite[Chapter 6]{CR}, for example.

\vspace{1mm}

Let $J(KG)$ denote the restricted direct product of $(K_vG)^\times$ with respect to the subgroups $(\mathcal{O}_{K_v}G)^\times$ for $v\in M_K$. We have a surjective homomorphism
\[j: J(KG) \longrightarrow \Cl(\mathcal{O}_KG);\hspace{1em} j(c) = [\mathcal{O}_KG\cdot c],\]
where we define
\[ \mathcal{O}_KG\cdot c =  \bigcap_{v\in M_K}\left(\mathcal{O}_{K_v}G\cdot c_v\cap KG\right).\]
This in turn induces an isomorphism
\begin{equation}\label{idelic description} \Cl(\mathcal{O}_KG) \simeq \frac{J(KG)}{(KG)^\times U(\mathcal{O}_KG)},\mbox{ where }U(\mathcal{O}_KG) = \prod\limits_{v\in M_K} (\mathcal{O}_{K_v}G)^\times.\end{equation}
Each component $(K_vG)^\times$ as well as $(KG)^\times$ also admit a Hom-description as follows. Write $\widehat{G}$ for the group of irreducible $K^c$-valued characters on $G$. We then have canonical identifications
\begin{align}\notag
(F^cG)^\times & = \mbox{Map}(\widehat{G}, (F^c)^\times) &(=\mbox{Hom}(\mathbb{Z}\widehat{G},(F^c)^\times)),\\\label{iden}
(FG)^\times &= \mbox{Map}_{\Omega_F}(\widehat{G},(F^c)^\times)&(=\mbox{Hom}_{\Omega_F}(\mathbb{Z}\widehat{G},(F^c)^\times)),
\end{align}
induced by the association $\alpha\mapsto(\chi\mapsto\alpha(\chi))$, where we define
\[\alpha(\chi) = \sum_{s\in G}\alpha_s \chi(s)\mbox{ for }\alpha=\sum_{s\in G}\alpha_s s.\]
Finally, we note that via (\ref{idelic description}) and (\ref{iden}), for each $k\in\bZ$ coprime to $|G|$, the $k$th Adams operation $\Psi_k$ on $\Cl(\mathcal{O}_KG)$ is induced by $\chi\mapsto \chi^k$ on $\widehat{G}$.

\subsection{McCulloh's characterization} \label{char section}

The characterization of $R(\mathcal{O}_KG)$ due to L. R. McCulloh \cite{McCulloh} is given in terms of the so-called \emph{Stickelberger transpose}. We shall recall its definition below.

\begin{definition}Let $G(-1)$ denote the group $G$ on which $\Omega_F$ acts by
\[\omega\cdot s= s^{\kappa(\omega^{-1})}\mbox{ for $s\in G$ and $\omega\in\Omega_F$},\]
where $\kappa(\omega^{-1})\in\bZ$, which is unique modulo $\exp(G)$, is such that 
\[\omega^{-1}(\zeta)=\zeta^{\kappa(\omega^{-1})}\mbox{ for all $\zeta\in F^c$ with $\zeta^{\exp(G)} = 1$}.\]
Note that if $\zeta_n\in F$, then $\Omega_F$ fixes all elements in $G(-1)$ of order dividing $n$.
\end{definition}

\begin{definition}\label{pair}Given $\chi\in\widehat{G}$ and $s\in G$, define 
\[\langle\chi,s\rangle \in \left\{\frac{0}{|s|},\frac{1}{|s|},\dots,\frac{|s|-1}{|s|} \right\}\mbox{ to be such that }\chi(s) = (\zeta_{|s|})^{|s|\langle\chi,s\rangle}.\]
Extend this to a pairing $\langle\hspace{1mm},\hspace{1mm}\rangle:\bQ\widehat{G}\times\bQ G\longrightarrow\bQ$ via $\bQ$-linearity, and define
\[ \Theta:\bQ\widehat{G}\longrightarrow\bQ G(-1);\hspace{1em}\Theta(\psi)=\sum_{s\in G}\langle\psi,s\rangle s,\]
called the \emph{Stickelberger map}.
\end{definition}

As shown in \cite[Proposition 4.5]{McCulloh}, the Stickelberger map
preserves the $\Omega_F$-action. Set $A_{\widehat{G}} = \Theta^{-1}(\mathbb{Z}G)$. Then, applying the functor $\mbox{Hom}(-,(F^c)^\times)$ and taking $\Omega_F$-invariants yield a homomorphism
\[\Theta^t:\mbox{Hom}_{\Omega_F}(\mathbb{Z}G(-1),(F^c)^\times)\longrightarrow \mbox{Hom}_{\Omega_F}(A_{\widehat{G}},(F^c)^\times);\hspace{1em}g\mapsto g\circ\Theta.\]
This is the \emph{Stickelberger transpose} map defined in \cite{McCulloh}.


\vspace{1mm}

For brevity, define
\begin{align}\label{lambda def} \Lambda(FG)^\times&=\mbox{Map}_{\Omega_F}(G(-1),(F^c)^\times) & (=\mbox{Hom}_{\Omega_F}(\mathbb{Z}G(-1),(F^c)^\times)).\end{align}
Observe that we have a diagram
\[\begin{tikzcd}[column sep = 1.75cm, row sep = 1.75cm](FG)^\times \arrow{r}{rag} & \mbox{Hom}_{\Omega_F}(A_{\widehat{G}},(F^c)^\times)\\&\Lambda(FG)^\times,\arrow{u}[right]{\Theta^t}
\end{tikzcd}\]
where $rag$ is restriction to $A_{\widehat{G}}$ via the identification (\ref{iden}). 

\vspace{1mm}

Now, let $J(\Lambda(KG))$ denote the restricted direct product of $\Lambda(K_vG)^\times$ with respect to the subgroups $\mbox{Map}_{\Omega_F}(G(-1),\mathcal{O}_{F^c}^\times)$ for $v\in M_K$. We then have the following partial characterization of $R(\mathcal{O}_KG)$; see \cite{McCulloh} for the full characterization.

\begin{lem}\label{char lem1}Given $c = (c_v) \in J(KG)$, if there exists $g = (g_v)\in J(\Lambda(KG))$ such that $rag(c_v) = \Theta^t(g_v)$ for all $v\in M_K$, then $j(c) \in R(\mathcal{O}_KG)$.
\end{lem}
\begin{proof}This follows directly from \cite[Theorem 6.17]{McCulloh}.
\end{proof}

For each $v\in M_K$, fix a uniformizer $\pi_v$ of $K_v$. We shall also need:

\begin{lem}\label{char lem2}Let $L/K$ be a tame and Galois extension with $\Gal(L/K)\simeq G$. Then, for each $v\in M_K$, there exists $s_v\in G$ whose order is the ramification index of $L/K$ at $v$, such that $\Xi_{-1}([\mathcal{O}_L]) = j(c_L)$, where $c_L = (c_{L,v}) \in J(KG)$ is defined by $c_{L,v}(\chi) = \pi_v^{\langle\chi,s_v\rangle + \langle\chi,s_v^{-1}\rangle}$ for $\chi\in\widehat{G}$ via the identification (\ref{iden}).
\end{lem}
\begin{proof}We have $L = KG\cdot b$ for some $b\in L$ by the Normal Basis Theorem, and since $\mathcal{O}_L$ is locally free over $\mathcal{O}_KG$ of rank one, for $v\in M_K$, we have
\[\mathcal{O}_{K_v}\otimes_{\mathcal{O}_K}\mathcal{O}_{L} = \mathcal{O}_{K_v}G\cdot a_v\mbox{ for some }a_v\in \mathcal{O}_{K_v}\otimes_{\mathcal{O}_K}\mathcal{O}_L.\]
Following the notation in \cite[Section 1]{McCulloh}, put
\[ \mathbf{r}_G(b) = \sum_{s\in G}s(b)s^{-1}\mbox{ and }\mathbf{r}_G(a_v) = \sum_{s\in G}s(a_v)s^{-1}.\]
Then, by \cite[Proposition 5.4]{McCulloh}, we may choose $a_v$ to be such that 
\begin{equation}\label{resolvent of a}\mathbf{r}_G(a_v)(\chi)=\pi_v^{\langle\chi,s_v\rangle}\mbox{ for all }\chi\in\widehat{G}.\end{equation}
By \cite[Proposition 3.2]{McCulloh} and the discussion following it, there exists
\[ c=(c_v)\in J(KG)\mbox{ such that }\mathbf{r}_G(a_v) = c_v\cdot \mathbf{r}_G(b)\mbox{ and }j(c) = [\mathcal{O}_L].\]
Write $[-1]$ for the involution on $(K_v^cG)^\times$ induced by the involution $s\mapsto s^{-1}$ on $G$. Since $\mathbf{r}_G(b)\mathbf{r}_G(b)^{[-1]}\in (KG)^\times$, we then deduce that
\[\Xi_{-1}([\mathcal{O}_L])= j((c_vc_v^{[-1]}))= j((\mathbf{r}_G(a_v)\mathbf{r}_G(a_v)^{[-1]})).\]
The claim now follows immediately from (\ref{resolvent of a}).
\end{proof}
 
\subsection{Generalized Swan subgroups}\label{Swan section}

Let $H$ be a subgroup of $G$. Following the definition of the Swan subgroup $T(\mathcal{O}_KG)$ given in \cite{Ullom Swan}, we shall define a \emph{generalized Swan subset/subgroup} associated to $H$ as follows.

\vspace{1mm}

For each $r\in\mathcal{O}_K$ coprime to $|H|$, define 
\[ (r,\Sigma_H) = \mathcal{O}_KG\cdot r + \mathcal{O}_KG\cdot \Sigma_H, \mbox{ where }\Sigma_H = \sum_{s\in H} s.\]
The next proposition, which generalizes \cite[Proposition 2.4 (i)]{Ullom Swan}, shows that $(r,\Sigma_H)$ is locally free over $\mathcal{O}_KG$ of rank one and so it defines a class $[(r,\Sigma_H)]$ in $\Cl(\mathcal{O}_KG)$. Define
\[ T_H^*(\mathcal{O}_KG) = \{[(r,\Sigma_H)]: r\in\mathcal{O}_K \mbox{ coprime to $|H|$}\}\]
to be the collection of all such classes. It follows directly from the definition that $T_G^*(\mathcal{O}_KG)$ is equal to $T(\mathcal{O}_KG)$.

\begin{prop}\label{class in T}Let $r\in\mathcal{O}_K$ be coprime to $|H|$. For each $v\in M_K$, define
\[c_{H,r,v} = \begin{cases}1&\mbox{if $v\nmid r$},\\ r + \frac{1-r}{|H|}\Sigma_H&\mbox{if $v\mid r$},\end{cases}\]
and set $c_{H,r} = (c_{H,r,v})$. Then we have $\mathcal{O}_{K}G\cdot c_{H,r} = (r,\Sigma_H)$.
\end{prop}
\begin{proof}For each $v\in M_K$, we need to show that 
\[\mathcal{O}_{K_v}G\cdot c_{H,r,v} = \mathcal{O}_{K_v}G\cdot r + \mathcal{O}_{K_vG}\cdot\Sigma_H.\]
For $v\nmid r$, we have $r\in\mathcal{O}_{K_v}^\times$, and this is clear.  For $v\mid r$, we have $|H|\in\mathcal{O}_{K_v}^\times$ because $r$ is coprime to $|H|$, and so
\[ \textstyle\mathcal{O}_{K_v}G\cdot\left(r + \frac{1-r}{|H|}\Sigma_H\right) \subset \mathcal{O}_{K_v}G\cdot r+ \mathcal{O}_{K_v}G\cdot \Sigma_H.\]
The reverse inclusion also holds because 
\[ \textstyle r = \left(1 + \frac{r-1}{|H|}\Sigma_H\right)\left(r +  \frac{1-r}{|H|}\Sigma_H\right)\mbox{ and }\Sigma_H = \Sigma_H\left(r +  \frac{1-r}{|H|}\Sigma_H\right).\]
We then see that the claim holds.
\end{proof}

In what follows, for simplicity, let us assume that 
\begin{equation}\label{ass}\mbox{$H$ is normal in $G$ and the quotient $G/H$ is abelian}.\end{equation}
Put $Q = G/H$, and let $H_1,\dots,H_q$ denote all the distinct cosets of $H$ in $G$. Notice that we have an augmentation homomorphism
\[\upepsilon:\mathcal{O}_KG\longrightarrow \mathcal{O}_KQ;
\hspace{1em}\upepsilon\left(\sum_{s\in G}\alpha_ss\right) = \sum_{i=1}^{q}\left(\sum_{s\in H_i}\alpha_s\right)H_i.\]
Then, we have a fiber product diagram of rings, given by
\begin{equation}\label{fiber product diagram}
\begin{tikzcd}[column sep = 1.75cm, row sep = 1.25cm]
\mathcal{O}_KG \arrow{r}{\upepsilon} \arrow[swap]{d} & \mathcal{O}_KQ \arrow{d}{\uppi}\\
\Gamma_H\arrow[swap]{r}{\overline{\upepsilon}}&\Lambda_{|H|}Q
\end{tikzcd}, \mbox{ where }
\begin{cases}\Gamma_H = \mathcal{O}_KG/(\Sigma_H),\\\Lambda_{|H|} = \mathcal{O}_K/|H|\mathcal{O}_K.\end{cases}\end{equation}
Here the vertical maps are the canonical quotient maps, and $\overline{\upepsilon}$ is the homomorphism induced by $\upepsilon$. We then have the identification 
\begin{equation}\label{fiber prod iden}\mathcal{O}_KG = \{(x,y) \in\mathcal{O}_KQ\times\Gamma_H: \uppi(x) = \overline{\upepsilon}(y)\}.\end{equation}
In particular, writing
\begin{equation}\label{xy notation} x = \sum_{i=1}^{q}x_iH_i,\,\ y = \widetilde{y} + (\Sigma_H),\,\ \upepsilon(\widetilde{y}) = \sum_{i=1}^{q}\widetilde{y}_i H_i, \end{equation}
the corresponding element in $\mathcal{O}_KG$ is given by
\[ \widetilde{y} + \left(\sum_{i=1}^q\left(\frac{x_i-\widetilde{y_i}}{|H|}\right)s_i\right)\Sigma_{H},\mbox{ where $s_i\in H_i$ is fixed}.\]
Since $Q$ is abelian, from the Mayer-Vietoris sequence (see \cite[Section 49B]{CR} or \cite[(1.12), (4.19), (4.21)]{RU}) associated to (\ref{fiber product diagram}), we obtain a homomorphism 
\[\partial_H: (\Lambda_{|H|}Q)^\times\longrightarrow D(\mathcal{O}_KG);\hspace{1em}\partial_H(\eta) = [(\mathcal{O}_KG)(\eta)],\]
where $D(\mathcal{O}_KG)$ denotes the kernel group in $\Cl(\mathcal{O}_KG)$ defined as in \cite{RU}, and
\[ (\mathcal{O}_KG)(\eta) = \{(x,y) \in \mathcal{O}_KQ \times \Gamma_H: \uppi(x) = \overline{\upepsilon}(y)\eta\}\]
is equipped with the obvious $\mathcal{O}_KG$-module structure via (\ref{fiber prod iden}).

\vspace{1mm}

The next proposition, which generalizes \cite[Proposition 2.7]{Ullom Swan}, shows that
\[T_H^*(\mathcal{O}_KG) = \partial_H(\Lambda_{|H|}^\times),\]
where $\Lambda_{|H|}$ is regarded as a subring of $\Lambda_{|H|}Q$ in the obvious way (cf. the set $T_H(\mathcal{O}_KG)$ defined in \cite{Oliver}). This means that under the assumption (\ref{ass}), the set $T_H^*(\mathcal{O}_KG)$ is in fact a subgroup of $\Cl(\mathcal{O}_KG)$.

\begin{prop}\label{T = Image}Let $r\in\mathcal{O}_K$ be coprime to $|H|$. Then we have \[\partial_H((r+|H|\mathcal{O}_K)H) = [(r,\Sigma_H)].\]
\end{prop}
\begin{proof} For brevity, put $\eta = (r+ |H|\mathcal{O}_K)H$. Note that by definition, we have
\[ \eta = \uppi(rH) = \overline{\upepsilon}(r+(\Sigma_H)).\]
Via the identification (\ref{fiber prod iden}), we may define an $\mathcal{O}_KG$-homomorphism
\[ \varphi: (\mathcal{O}_KG)(\eta) \longrightarrow \mathcal{O}_KG;\hspace{1em}
\varphi(x,y) = (x,y(r + (\Sigma_H))).\]
Below, we shall show that $\Image(\varphi) = (r,\Sigma_H)$ and $\ker(\varphi) = 0$. This would imply that $(\mathcal{O}_KG)(\eta)$ and $(r,\Sigma_H)$ are isomorphic as $\mathcal{O}_KG$-modules, from which the claim follows. Given $(x,y)\in(\mathcal{O}_KG)(\eta)$, in the notation of (\ref{xy notation}), we have
\begin{equation}\label{image phi} (x,y(r + (\Sigma_H))) = \widetilde{y}r + \left(\sum_{i=1}^{q}\left(\frac{x_i - \widetilde{y_i}r}{|H|}\right)s_i\right)\Sigma_{H}.\end{equation}
First, from (\ref{image phi}), we immediately see that $\Image(\varphi) \subset (r,\Sigma_H)$, as well as
\[ \varphi((rH,1+(\Sigma_H)) = r
\mbox{ and }\varphi(|H|H,(\Sigma_H)) = \Sigma_H,\]
whence $\Image(\varphi) \supset (r,\Sigma_H)$ holds also. Next, suppose that $(x,y)\in\ker(\varphi)$. It is clear from the definition of $\varphi$ that $x= 0$. Then, we deduce from (\ref{image phi}) that
\[ \widetilde{y}r - \left(\sum_{i=1}^{q}\frac{\widetilde{y}_ir}{|H|}s_i\right)\Sigma_H = 0\mbox{ and hence } \widetilde{y} \in (\Sigma_H).\]
This shows that $y = 0$, and so $\ker(\varphi) = 0$, as desired.
\end{proof}

\subsection{Preliminaries} Let $H$ be a subgroup of $G$ and let $r\in\mathcal{O}_K$ be coprime to $|H|$. Then, via the isomorphism (\ref{idelic description}), we have
\begin{equation}\label{j(cr)} j(c_{H,r}) = [(r,\Sigma_H)],\mbox{ where }(c_{H,r}) = (c_{H,r,v})\in J(KG)\end{equation}
is defined as in Proposition~\ref{class in T}. Also, note that for $v\mid r$, we have
\begin{equation}\label{cr chi} c_{H,r,v}(\chi) =\begin{cases} 1 & \mbox{if $\chi(H) = 1$} \\ r &\mbox{if $\chi(H)\neq1$}\end{cases}\end{equation}
for $\chi\in\widehat{G}$ via the identification (\ref{iden}). This immediately implies that:

\begin{prop}\label{T invar}We have $T_H^*(\mathcal{O}_KG)\subset \Cl(\mathcal{O}_KG)^{\Psi_\bZ}$.
\end{prop}
\begin{proof}This follows from (\ref{cr chi}) and the fact that
\[ \chi^k(H) = 1\mbox{ if and only if }\chi(H) = 1\]
for any $k\in\bZ$ coprime to $|H|$.
\end{proof}

To make connections between $T_H^*(\mathcal{O}_KG)$ and $R(\mathcal{O}_KG)$, we shall use Lemmas~\ref{char lem1} and~\ref{char lem2}. We shall also need the following definitions.

\vspace{1mm}

Fix a prime $v\in M_K$. Recall from (\ref{iden}) and (\ref{lambda def}) that
\[ (K_vG)^\times = \mbox{Map}_{\Omega_{K_v}}(\widehat{G},(K_v^c)^\times)\mbox{ and }\Lambda(K_vG)^\times = \mbox{Map}_{\Omega_{K_v}}(G(-1), (K_v^c)^\times).\]
Given $t\in G$ with $t\neq1$ and $x\in K_v^\times$, define
\[c_{t,v,x,1}(\chi) = x^{\langle\chi ,t\rangle + \langle\chi,t^{-1}\rangle}\mbox{ and }c_{t,v,x,2}(\chi)= x^{2\langle\chi,t\rangle-\langle\chi,t^2\rangle}\]
for $\chi\in\widehat{G}$, where both exponents are integers by Definition~\ref{pair}. In the case that $|t| = 2$ and $|t|>2$, respectively, define
\[g_{t,v,x,1}(s) = \begin{cases} x^2 & \mbox{if $s = t$}\\
1&\mbox{otherwise}\end{cases}\mbox{ and }g_{t,v,x,1}(s) = \begin{cases}x &\mbox{for $s\in\{t,t^{-1}\}$}\\1&\mbox{otherwise}\end{cases}\]
for $s\in G(-1)$. In the case that $|t|$ is odd, further define
\[g_{t,v,x,2}(s) = \begin{cases} x^2 & \mbox{if $s = t$}\\ x^{-1}&\mbox{for $s = t^{2}$}\\ 1&\mbox{otherwise}\end{cases}\]
for $s\in G(-1)$. We have the following lemmas.

\begin{lem}\label{char lem3'} We have $c_{t,v,x,1}\in(K_vG)^\times$.
\end{lem}
\begin{proof}The map $c_{t,v,x,1}$ preserves the $\Omega_{K_v}$-action because 
\begin{equation}\label{pair sum} \langle\chi,s\rangle + \langle\chi,s^{-1}\rangle = \begin{cases}0&\mbox{if $\chi(s) = 1$}\\
1&\mbox{if $\chi(s)\neq1$}\end{cases}\end{equation}
for all $\chi\in\widehat{G}$ and $s\in G$ by Definition~\ref{pair}.
\end{proof}

\begin{lem}\label{char lem3}Suppose that $\zeta_{|t|}\in K_v^\times$. Then we have 
\[c_{t,v,x,2}\in (K_vG)^\times\mbox{ and }g_{t,v,x,1},g_{t,v,x,2}\in \Lambda(K_vG)^\times.\]
Moreover, for both $i=1,2$, we have 
\[rag(c_{t,v,x,i}) = \Theta^t(g_{t,v,x,i}).\]
\end{lem}
\begin{proof}Since $\zeta_{|t|}\in K_v^\times$, we easily see that $c_{t,v,x,2}$, $g_{t,v,x,1}$, and $g_{t,v,x,2}$ indeed all preserve the $\Omega_{K_v}$-action. Since
\begin{equation}\label{Theta g} \Theta^t(g)(\psi) = \prod_{s\in G}g(s)^{\langle\psi,s\rangle}\mbox{ for $g\in\Lambda(K_vG)^\times$ and $\psi\in A_{\widehat{G}}$}\end{equation}
by definition, the second also holds by a simple verification.
\end{proof}

\subsection{Proof of Theorem~\ref{thm Swan1}} Let $H$ be a cyclic subgroup of $G$ of order $n$ and let $r\in\mathcal{O}_K$ be coprime to $n$. Recall (\ref{j(cr)}) and that $j(c_{H,r})\in \Cl(\mathcal{O}_KG)^{\Psi_\bZ}$ by Proposition~\ref{T invar}. We need to show that $j(c_{H,r})^{d_n(K)}\in R(\mathcal{O}_KG)$ in part (a), and that $j(c_{H,r})\in\Image(\Xi_{-1})$ in part (b). We shall do so using Lemma~\ref{char lem1}. 

\vspace{1mm}

In what follows, let $t$ be a fixed generator of $H$.

\begin{proof}[Proof of Theorem~\ref{thm Swan1} (a)]
Let $D$ be the subgroup of $(\bZ/n\bZ)^\times$ such that
\[ D\simeq\Gal(K(\zeta_n)/K)\mbox{ via }i \mapsto (\zeta_n\mapsto \zeta_n^i).\]
Define $g = (g_v)\in J(\Lambda(KG))$ by setting $g_v = 1$ for $v\nmid r$, and 
\[ g_v(s) = \begin{cases} r &\mbox{if $s \in\{ t^i, t^{-i}\}$ for some $i\in D$}\\ 1&\mbox{otherwise}\end{cases}\]
for $v\mid r$. It is easy to see that $g_v$ preserves the $\Omega_{K_v}$-action.

\vspace{1mm}

Observe that for all $v\in M_K$, we have
\begin{equation}\label{rag theta}rag((c_{H,r,v})^{d_n(K)}) =\Theta^t(g_v).\end{equation}
Indeed, for $v\nmid r$, this is clear. As for $v\mid r$, we have from (\ref{Theta g}) that
\[\Theta^t(g_v)(\chi) = \begin{cases}
\left(r\right)^{\frac{1}{2}\sum\limits_{i\in D}\left(\langle\chi,t^i\rangle + \langle\chi,t^{-i}\rangle\right)} & \mbox{if $-1\in D$}\\
\left(r\right)^{\sum\limits_{i\in D}\left(\langle\chi,t^i\rangle + \langle\chi,t^{-i}\rangle\right)} & \mbox{if $-1\notin D$}
\end{cases}\]
and from (\ref{pair sum}) that
\[\sum_{i\in D}\left(\langle\chi,t^i\rangle + \langle\chi,t^{-i}\rangle\right) = \begin{cases} 0 &\mbox{if $\chi(t) = 1$} \\ |D| &\mbox{if $\chi(t)\neq1$}\end{cases}\] 
for any $\chi\in\widehat{G}$. The equality (\ref{rag theta}) then follows from (\ref{cr chi}). Hence, we have $j(c_{H,r})^{d_n(K)}\in R(\mathcal{O}_KG)$ by Lemma~\ref{char lem1}, as desired.
\end{proof}

\begin{proof}[Proof of Theorem~\ref{thm Swan1} (b)] Suppose that $n$ is odd and that $\zeta_n\in K^\times$. Then, by Lemma~\ref{char lem3}, we may define $c = (c_v) \in J(KG)$ by setting $c_v = 1$ for $v\nmid r$, and $c_v = c_{t,v,r,2}$ for $v\mid r$. Also, we have $j(c)\in R(\mathcal{O}_KG)$ by Lemma~\ref{char lem1}.

\vspace{1mm}

Below, we shall show that $j(c_{H,r}) = \Xi_{-1}(j(c))$, whence $j(c_{H,r})\in\Image(\Xi_{-1})$. To that end, let $v\in M_K$ and $\chi\in\widehat{G}$. It suffices to show that
\begin{equation}\label{ccc} c_{H,r,v}(\chi) = c_v(\chi)c_v(\chi^{-1}).\end{equation}
For $v\nmid r$, this is clear. For $v\mid r$, observe that 
\[\chi(t) = 1\mbox{ if and only if }\chi(t^2) = 1\]
 because $|t|$ is odd. It then follows from (\ref{pair sum}) that
\begin{align*}c_v(\chi)c_v(\chi^{-1}) &= r^{2(\langle\chi,t\rangle + \langle\chi^{-1},t\rangle) - (\langle\chi,t^2\rangle + \langle\chi^{-1},t^2\rangle)} = \begin{cases} 1 &\mbox{for $\chi(t) = 1$},\\r&\mbox{for $\chi(t) \neq 1$}.\end{cases}\end{align*}
From (\ref{cr chi}), we then see that (\ref{ccc}) indeed holds.
\end{proof}

\subsection{Proof of Theorem~\ref{thm Swan2}} In what follows, for each $v\in M_K$, let $\pi_v$ be a fixed uniformizer of $K_v$. 

\begin{proof}[Proof of Theorem~\ref{thm Swan2} (a)] Suppose that $\Cl(\mathcal{O}_K)=1$. For each $v_0\in M_K$, we may then choose $\pi_{v_0}$ to be an element of $\mathcal{O}_K$. Then, for any cyclic subgroup $H$ of $G$ of order coprime to $v_0$, it makes sense to write 
\[[(\pi_{v_0},\Sigma_H)] = j(c_{H,\pi_{v_0}}),\mbox{ where }c_{H,\pi_{v_0}} = (c_{H,\pi_{v_0},v}) \in J(KG)\]
is as in Proposition~\ref{class in T}. Plainly $c_{H,\pi_{v_0},v} = 1$ for all $v\neq v_0$.

\vspace{1mm}
 
Now, let $L/K$ be any tame and Galois extension with $\Gal(L/K)\simeq G$, and we shall use the notation as in Lemma~\ref{char lem2}. Let $V$ denote the subset of $M_K$ consisting of the primes which ramify in $L/K$. Then, we have
\[s_v=1\mbox{ for $v\notin V$, and so }\Xi_{-1}([\mathcal{O}_L]) = j(c_L) = \prod_{v_0\in V} j(c_{L,v_0}),\]
where we regard $c_{L,v_0}$ as an element of $J(KG)$ whose components outside of $v_0$ are all $1$. For each $v_0\in V$, take $H_{v_0} = \langle s_{v_0}\rangle$, whose order is coprime to $v_0$ because $L/K$ is tame. By (\ref{cr chi}) and (\ref{pair sum}), we have
\[ c_{L,v_0}(\chi) = \pi_{v_0}^{\langle\chi,s_{v_0}\rangle +\langle\chi,s_{v_0}^{-1}\rangle} = c_{H_{v_0},\pi_{v_0},v_0}(\chi)\mbox{ for all }\chi\in\widehat{G}.\]
By definition, we also have $c_{L,v_0,v} = c_{H_{v_0},\pi_{v_0},v} = 1$ for $v\neq v_0$. It follows that $j(c_{L,v_0}) = j(c_{H_{v_0},\pi_{v_0}})$, which is an element of $T_{H_{v_0}}^*(\mathcal{O}_KG)$. This implies that
\[ \Xi_{-1}([\mathcal{O}_L]) \in \prod_{v_0\in V}T_{H_{v_0}}^*(\mathcal{O}_KG) \subset T_{\cyc}^*(\mathcal{O}_KG),\vspace{-1mm}\]
as claimed.
\end{proof}

\begin{proof}[Proof of Theorem~\ref{thm Swan2} (b)]Suppose that $G\neq1$. Then, fix an element $t\in G$ with $t\neq1$, whose order shall be assumed to be odd when $\delta(G) = 1$, and fix a character $\chi\in\widehat{G}$ such that $\chi(t)\neq1$. Now, suppose that $\zeta_{\exp(G)}\in K^\times$. Then, via (\ref{idelic description}) and (\ref{iden}), evaluation at $\chi$ induces a surjective homomorphism
\[\xi_\chi : \Cl(\mathcal{O}_KG) \longrightarrow \Cl(\mathcal{O}_K).\]
Below, we shall show that 
\begin{equation}\label{bound} \xi_{\chi}(\Image(\Xi_{-1})) \supset \Cl(\mathcal{O}_K)^{\delta(G)},\end{equation}
from which the claim would follow.

\vspace{1mm}

Now, every class in $\Cl(\mathcal{O}_K)$ may be represented by a prime ideal $\mathfrak{p}_0$ in $\mathcal{O}_K$, corresponding to $v_0\in M_K$, say. Since $\zeta_{\exp(G)}\in K^\times$, by Lemmas~\ref{char lem3'} and~\ref{char lem3}, we may define $c = (c_v)\in J(KG)$ by setting 
\[c_{v_0} =\begin{cases} c_{t,v,\pi_{v_0},1} & \mbox{if $\delta(G) = 2$} \\ c_{t,v,\pi_{v_0},2} & \mbox{if $\delta(G) = 1$}\end{cases}\]
and $c_v = 1$ for $v\neq v_0$. Note that $j(c)\in R(\mathcal{O}_KG)$ by Lemmas~\ref{char lem1} and~\ref{char lem3}, whence $\Xi_{-1}(j(c)) \in \Image(\Xi_{-1})$. Also, we have
\begin{align*} c_{v_0}(\chi)c_{v_0}(\chi^{-1}) &=
\begin{cases}
 \pi_{v_0}^{2(\langle\chi,t\rangle+\langle\chi^{-1},t\rangle)} & \mbox{if $\delta(G) =2$}\\
\pi_{v_0}^{2(\langle\chi,t\rangle + \langle\chi^{-1},t\rangle) - (\langle\chi,t^2\rangle + \langle\chi^{-1},t^2\rangle)} & \mbox{if $\delta(G) = 1$}
\end{cases}\\& = \pi_{v_0}^{\delta(G)} \end{align*}
by (\ref{pair sum}). We then deduce that 
\[\xi_\chi(\Xi_{-1}(j(c))) = [\mathfrak{p}_0]^{\delta(G)}\mbox{ in }\Cl(\mathcal{O}_K).\]
This proves the desired inclusion (\ref{bound}). 
\end{proof}

\section{Acknowledgments}

The research was supported by the China Postdoctoral Science  Foundation Special Financial Grant (grant no.: 2017T100060). The author would like to thank the anoynmous referee for pointing out some unclear arguments in the original manuscript and for suggesting the reference \cite{Lang} cited in Lemma~\ref{group1}.

\end{document}